\documentclass[proceedings%
]{aofa}

\usepackage[utf8]{inputenc}
\usepackage{subfigure}

\newtheorem{theorem}{Theorem}
\newtheorem{lemma}{Lemma}
\newtheorem{example}{Example}

\usepackage[round]{natbib}

\author[M. Drmota and M. Javanian]{Michael Drmota\addressmark{1}\thanks{Partially supported by the Austrian Science Fund FWF, Project SFB F50-02}\and Mehri Javanian}

\title[First Order Linear Partial Differential Equations Related to Urn Models]{Solutions of First Order Linear Partial Differential Equations Related to Urn Models and Central Limit Theorems}

\address{\addressmark{1}TU Wien, Institute of Discrete Mathematics and Geometry, 
Wiedner Hauptstrasse 8--10, A-1040 Wien, Austria\\
\addressmark{2}University of Zanjan, Department of Statistics, University Blvd., 45371-38791, Zanjan, Iran}
\keywords{urn models, first order partial differential equations, central limit theorem, singularity analysis}

\begin{document}
\maketitle
\begin{abstract}
\paragraph{Abstract.}
We study first order linear partial differential equations that appear, for example,  in the 
analysis of dimishing urn models with the help of the method of characteristics and
formulate sufficient conditions for a central limit theorem.
 \end{abstract}

\section{Introduction and Main result}

The purpose of this paper is to study solutions $H(z,w)$ of special 
first order linear partial differential equations that appear in the
analysis of dimishing urn models. In particular we follow the work of 
\cite{HKP}.

More precisely, we consider a P\'olya-Eggenberger urn model
with two kinds of balls and transition matrix
$M = \left( \begin{array}{cc} a & b \\ c & d \end{array} \right)$.
The process runs as follows. Suppose that the urn contains $m$ balls of the first kind
and $n$ balls of the second kind - we can interprete this state as the 
point $(m,n)$ on the integer lattice.
Then with probability $m/(n+m)$ we add $a$ balls of the first
kind and $b$ balls of the second kind, whereas with probability $n/(n+m)$ we add $c$ balls of the first
kind and $d$ balls of the second kind. (Of course, adding a negative number of balls means taking
away this number of balls.)
An absorbing state $\mathcal{S}$ is a subset $\mathcal{S}\subset \mathbb{N}\times \mathbb{N}$,
where the process stops when we arrive in $\mathcal{S}$. In what follows we will only consider
(special) dimishing urn-models, where the number of balls of the first kind eventually reaches
zero, so that the $y$-axis $\mathcal{S} = \{ (0,n) : n\ge 0 \}$ is a natural absorbing state. 

Suppose now that the process starts at $(m,n)\in \mathbb{N}\times \mathbb{N}$ with $m\ge 1$ and 
let $h_{n,m}(v) = \mathbb{E}[v^{X_{n,m}}]$ denote the probability generating function of the random variable
$X_{n,m}$ that describes the position $(0,n_0)$ of the absorbing state in  $\mathcal{S}$
when the process starts at $(m,n)$.

By definition the probability generating functions $h_{n,m}(v)$ satisfy the recurrence
\begin{equation}\label{eqhrec}
h_{n,m}(v) = \frac n{n+m} h_{n+a,m+c}(v) + \frac m{n+m} h_{n+c,m+d}(v).
\end{equation}
for $(m,n)\not\in \mathcal{S}$. The boundary values at an absorbing state $(m,n)\in \mathcal{S}$
is $h_{n,m}(v) = v^n$.

By setting 
\[
\overline H(z,w;v) = \sum_{n\ge 0,m\ge 1} h_{n,m}(v) z^n w^m
\]
it follows that this generating function $\overline H(z,w;v)$ satisfies the 
partial differential equation
\begin{equation}\label{eqPDE0}
z(1-z^{-a}w^{-b}) \overline H_z + w(1-z^{-c}w^{-d}) \overline H_w + (az^{-a}w^{-b} + dz^{-c}w^{-d}) \overline H = F(z,w)
\end{equation}
with some inhomogeneous part $F(z,w)$ that is given by the boundary values
which are partly unknown (for example $\overline H(0,w,v)$, see \cite{HKP}.

We want to mention that first order linear partial differential equations related to urn models were first systematically
discussed by \cite{FlaGabPek}, see also \cite{Morcrette}, where a special case is detailly treated.
On the other hand, it is possible to describe the probabilistic behavior of the development of urn models
very precisely, sse \cite{Janson1,Janson2}, even with absorbing states. Nevertheless the analysis of 
dimishing urns with the $y$-axis as the absorbing state is still quite special. Here we also refer
to \cite{Kuba,KubaPanholzer}, where the analysis is based directly on the recurrence (\ref{eqhrec}).
Another interesting paper that is related to dimishing urn models and lines as absorbing states is
\cite{KubaPanProd}. There the authors observe several different kinds of limiting behaviors (with 
five phase changes).

It turns out that there are some special cases, where it is more convenient to study the generating function
\begin{equation}\label{eqHdef}
H(z,w;v) = \sum_{n\ge 0, m\ge 1} {n+m \choose m} h_{n,m}(v) z^n w^m
\end{equation}
that (also) satisfies a first order linear partial differential equation of the form
\begin{equation}\label{eqPDE}
A(z,w) H_z + B(z,w) H_w - C(z,w) H = D(z,w;v),
\end{equation}
with analytic functions $A(z,w), B(z,w), C(z,w), D(z,w;v)$. (In the examples below 
$A(z,w), B(z,w)$, and $C(z,w)$ are polynomials.)
For these particular cases it turns out that the unknown boundary conditions are not needed
since they cancel in the equation. Nevertheless the methods that we are developing below are 
-- although we do not work out the general case --
suitable to deal with equations of the form (\ref{eqPDE0}). 

Note that by definition
\begin{equation}\label{eqHw0}
H(z,0;v) = 0.
\end{equation}
Furthermore, if $v=1$ then $h_{n,m}(1) = 1$ so that 
\[
H(z,w;1) = \frac 1{1-z-w} -\frac 1{1-z}.
\]
This means that
$D(z,w;1)$ is determined by
\[
D(z,w;1) = \frac{A(z,w) + B(z,w) - (1-z-w)C(z,w)}{(1-z-w)^2} - \frac{A(z,w)- (1-z)C(z,w)}{(1-z)^2}.
\]
In the present context it is convenient to assume that the function
\[
H_0(z,w) = \sum_{n,m\ge 0} {n+m \choose m}  z^n w^m = 1/(1-z-w)
\]
is a solution
of the homogeneous differential equation $A(z,w) H_z + B(z,w) H_w - C(z,w) H = 0$ so that 
\begin{equation}\label{eqcond1}
A(z,w) + B(z,w) = (1-z-w)C(z,w)
\end{equation}
and, thus, 
\begin{equation}\label{eqcond2}
D(z,w;1) = -\frac{A(z,w)- (1-z)C(z,w)}{(1-z)^2}
\end{equation}

We first state the following three examples from \cite{HKP} (that we present in
a slightly modified way).

\begin{example}{\rm 
The pill's problem (see \cite{BP,Kn}) has transition matrix $M = \left( \begin{array}{cc} -1 & 0 \\ 1 & -1 \end{array} \right)$
and absorbing state $\mathcal{S} = \{ (0,n) : n\ge 0 \}$, and the corresponding differential equation 
is given by
\[
(z-z^2-w)H_z + w(1-z)H_w - zH = \frac{w v}{(1-vz)^2}.
\]
Here it follows that $h_{n,m}(v)$ is given by 
\[
h_{n,m}(v) = mv \int_0^1 (1+(v-1)q)^n (1-q - (v-1)q \log q)^{m-1}dq.
\]
Finally the corresponding random variable $X_{n,m}$ has limiting distribution
\[
\frac{X_{n,m}}{ \frac nm + \log m} \to X \qquad (m\to\infty),
\]
where $X$ has density $e^{-x}$, $x\ge 0$, or
\[
\frac{X_{n,m}}n \to {\rm Beta}(1,m) \qquad (\mbox{fixed $m\ge 1$, $n\to\infty$}),
\]
where (the beta distribution) ${\rm Beta}(1,m)$ has density $m(1-x)^{m-1}$, $0\le x\le 1$.
}
\end{example}

\begin{example}{\rm 
A variant of the pill's problem has transition matrix $M = \left( \begin{array}{cc} -1 & 0 \\ 1 & -2 \end{array} \right)$
and absorbing state $\mathcal{S} = \{ (0,n) : n\ge 0 \}\cup \{ (1,n) : n\ge 0 \}$.
Due to the parity condition in $m$ (that is, only even $m$ occur), it is convenient to consider the generating function
\[
H(z,w;v) = \sum_{n\ge 0,m\ge 1} {n+2m \choose n} h_{n,2m}(v) z^n w^m
\]
that satisfies the differential equation
\[
-w H_z + 2w(1-z)H_w - (1-z)H = \frac{wv}{(1-vz)^2}.
\]
Here we obtain an explicit solution of the form
\begin{eqnarray*}
H(z,w) &=& \frac{w}{v\left( (1-z)^2-w - ((v-1)/v))^2 \right)\left( 1 - z - (v-1/v) \right) }  \\
&+& \frac{(v-1) \sqrt w}{v^2 \left( (1-z)^2-w - ((v-1)/v))^2 \right)^{3/2} }
\arctan \left( \frac  {\sqrt{w} \sqrt{(1-z)^2-w - ((v-1)/v))^2} }  { (1-z)^2-w - (1-z)(v-1)/v) }     \right).
\end{eqnarray*}
which leads to the limiting behavior:
\[
\frac{X_{n,2m}}{\frac{n}{\sqrt m} + 2 \sqrt m} \to R \qquad (m\to\infty),
\]
where $R$ has density $2xe^{-x^2}$, $x\ge 0$, or
\[
\frac{X_{n,2m}}n \to \sqrt{{\rm Beta}(1,m)}, \qquad (\mbox{$m\ge 1$ fixed, $n\to\infty$}).
\]
}
\end{example}

\begin{example}{\rm 
The cannibal urn (see \cite{Pittel,Kuba}) 
has transition matrix $M = \left( \begin{array}{cc} 0 & -1 \\ 1 & -2 \end{array} \right)$
and absorbing state $\mathcal{S} = \{ (0,n) : n\ge 0 \}\cup \{ (1,n) : n\ge 0 \}$ and the
generating function
\[
H(z,w;v) = \sum_{n\ge 0,m\ge 1} {n+m \choose n} h_{n+1,m}(v) 
\]
satisfies the differential equation
\[
-(z+w) H_z + H_w - H = \frac{(1+wv)v}{(1-vz)^2}.
\]
The solution is explicitly given by
\[
H(z,w;v) = \frac {ve^w}{1 - (1-e^w(1-z-w))v} - \frac{v}{1-vz}
\]
and we have a central limit theorem of the form
\[
\frac{X_{n,m} - \mathbb{E}\, X_{n,m} }{\sqrt{\mathbb{V}{\rm ar} X_{n,m}   }} \to N(0,1) \qquad
(m+n\to\infty).
\]
}
\end{example}

These three examples show that although the linear differential equations look
very similar the limiting behavior of the encoded random variable $X_{n,m}$ 
seems to be far from being universal. The main purpose of the present paper
is to shed some light on this phenomenon. In particular we detect a sufficient condition that
ensures a central limit theorem.

\begin{theorem}\label{Th1}
Suppose that $X_{n,m}$, $n\ge 0$, $m\ge 1$ are non-negative discrete random
variables with probabilty generating function $h_{n,m}(v) = \mathbb{E}[ v^{X_{n,m}} ]$
such that the generating function $H(z,w;v)$, given by (\ref{eqHdef}) satisfies a
first order linear differential equation of the form (\ref{eqPDE}), where 
the coefficient functions $A,B,C$ as well as the ratios $A(z,w)/B(z,w)$, $C(z,w)/B(z,w)$ 
are analytic in an open set that contains $z,w$ with $|z|+ |w|\le 1$ such that
the ratio $A(z,w)/B(z,w)$ is negative for non-negative $z,w$.
Furthermore we assume that (\ref{eqcond1}) is satisfied (which also implies (\ref{eqcond2})) and that
$D(z,w;v)$ can be represented as 
\[
D(z,w;v) = \frac{a(z,w;v)}{(1- b(z,w;v))^2},
\]
where the functions $a,b$ are also in an open set that contains $z,w$ with $|z|+ |w|\le 1$.
In particular in accordance with (\ref{eqcond1}) we have $a(z,w;1) = -A(z,w) + (1-z)C(z,w)$ and
$b(z,w;1) = z$.

Let $f(c,s)$ be the solution of the 
differential equation $\frac{\partial f}{\partial s} = A(f,s)/B(f,s)$ with $f(c,0)= c$ and
let $Q(z,w)$ denote the function that satisfies $f(Q(z,w),w) = z$.
We further assume that the function $f(Q(z,w),s)$ is analytic in an open set that contains $s,z,w$ with 
$|z|+ |w|\le 1$ and $|z|+ |s|\le 1$ and non-decreasing for positive and real
$z$ and $w$, 

Let $z_0(\rho;v)$ and $w_0(\rho;v)$ denote the solutions of the system of equations
\[
b(f(Q(z,w),0),0;v) = 1, \qquad  z \frac{\partial}{\partial z} b(f(Q(z,w),0),0;v) = 
\rho w \frac{\partial}{\partial w} b(f(Q(z,w),0),0;v)
\]
with $z_0(\rho;1) = \rho/(1+\rho)$ and $w_0(\rho;1) = 1/(1+\rho)$.
Furthermore set $h(\rho;v) = - \log z_0(\rho;v) - \rho \log w_0(\rho;v)$, 
$\mu(\rho) = \left.\frac{\partial}{\partial v} h(\rho;v)\right|_{v=1}$ and
$\sigma^2(\rho) = \left.\frac{\partial^2}{\partial v^2} h(\rho;v)\right|_{v=1} + \mu$.
If 
\[
\mu(\rho)  > 0 \qquad \mbox{for $\rho \in [\alpha,\beta]$}
\]
for some positive $\alpha,\beta$
then $X_{n,m}$ satisfies a central limit theorem of the form
\[
\frac{X_{n,m} - \mathbb{E}\, X_{n,m} }{\sqrt{ n }} \to N(0,\sigma^2(m/n)) \qquad
\]
uniformly for $m+n\to\infty$,  $m/n \in [\alpha,\beta]$, 
where 
\[
\mathbb{E}\, X_{n,m} \sim \mu(m/n)\, n \qquad\mbox{and}\quad
\mathbb{V}{\rm ar} X_{n,m} \sim \sigma^2(m/n)\, n.
\]
\end{theorem}

This theorem does not provide a full answer to the problem. However, it 
is a first step that covers at least a part, where we obtain a central limit theorem.
In future work we will provide a more complete picture, also covering the
cases, where there is no central limit theorem. For example it is not clear whether
it is possible to formulate conditions that refer directly to the entries of the
transition matrix
$M = \left( \begin{array}{cc} a & b \\ c & d \end{array} \right)$.
In particular it is an open question whether it is possible to adapt Theorem~\ref{Th1}
so that all cases of (\cite{Kuba}) are covered.

Nevertheless, we will discuss the three examples (from above) and another one in the next section.
We also present a (short version of the) proof of Theorem~\ref{Th1} in the remaining
parts of the paper.

\section{Discussion of Examples}

We do not work out the details here but in {\bf Examples 1} and {\bf 2} several
conditions of Theorem~\ref{Th1} are not satisfied, in particular we have
$\mu(\rho) = 0$.

\bigskip

The most interesting example is {\bf Example 3}.  Here we have $A(z,w) = -z-w$, $B(z,w) = C(z,w)= 1$, and
$D(z,w;v) = (1+wv)v/(1-vz)^2$, that is $a(z,w;v) = (1+wv)v$ and $b(z,w;v) = vz$.
It is easy to check that the conditions of Theorem~\ref{Th1} are satisfied.

In particular it follows that $f(c,s) = 1- s - e^{-s}(1-c)$,
$Q(z,w) = 1- e^w(1-z-w)$, and $f(Q(z,w),s) = 1-s - e^{w-s}(1-z-w)$.
From that we obtain $b(f(Q(z,w,0),0;v) = (1 - e^w(1-z-w))v$. Hence
the functions $z=z_0(\rho;v)$ and $w=w_0(\rho;v)$ satisfy the system of equations
\[
(1 - e^w(1-z-w))v = 1, \qquad z = \rho w(z+w)
\]
from which we obtain (by implicit differentiation)
\[
\mu(\rho)  = -\frac {z_{0,v}(\rho;1)}{z_0(\rho;1)} - \rho \frac{w_{0,v}(\rho;1)}{w_0(\rho;1)}
= 2 e^{-1/(1+\rho)} > 0.
\]
Thus, the central limit theorem follows automatically.

\bigskip

We add a new example in order to demonstrate the applicabilty of Theorem~\ref{Th1} 
(even if this example is not related to an urn model).
By the way this example can be easily generalized. 
Suppose that $H(z,w;z)$ satisfies the differential equation
\[
-(z+2w)H_z + (1+w)H_w - H = \frac{(1+2w)v}{(1-vz)^2}.
\]
Then again all assumptions of Theorem~\ref{Th1} are satisfied. 
Here we have  $A(z,w) = -z-2w$, $B(z,w) = 1+w$, $C(z,w)= 1$, and
$D(z,w;v) = (1+2w)v/(1-vz)^2$, that is $a(z,w;v) = (1+2w)v$ and $b(z,w;v) = vz$.

From this it follows that 
\[
f(c,s) = \frac{c}{1+s} - \frac{s^2}{1+s} \quad\mbox{and}\quad
Q(z,w) = (1+w)z + w^2
\]
and consequently
\[
f(Q(z,w),s) = \frac{(1+w)z + w^2-s^2}{1+s}.
\]
The functions $z=z_0(\rho;v)$ and $w=w_0(\rho;v)$ satisfy the system of equations
\[
((1+w)z+w^2)v = 1, \qquad z(1+w) = \rho w(z+2w)
\]
from which we obtain (by implicit differentiation)
\[
\mu(\rho)  = -\frac {z_{0,v}(\rho;1)}{z_0(\rho;1)} - \rho \frac{w_{0,v}(\rho;1)}{w_0(\rho;1)}
= 2 \frac{(1+\rho)^2}{(2+\rho)^2} > 0.
\]
Thus, the central limit theorem follows (again) automatically.

\section{The method of characteristics}

The first step of the proof is to use the theory of characteristics to provide an integral 
representation (\ref{eqHrep}) of the solution of the partial differential equation (\ref{eqPDE}).

We start with the inhomogeneous differential equation (\ref{eqPDE}), where $v$ is considered as
a parameter. It is a standard procedure to transform (\ref{eqPDE}) into 
a homogeneous equation. Let $Q = Q(z,w,H;v)$ denote the solution of the linear differential equation
\begin{equation}\label{eqPDE2}
A(z,w) Q_z + B(z,w) Q_w + (C(z,w)H + D(z,w;v)) Q_H = 0.
\end{equation}
Then the solution $H(z,w;v)$ of the original equation (\ref{eqPDE}) satisfies
the implicit equation 
\begin{equation}\label{eqQHequ}
Q(z,w,H(z,w;v);v) = const.
\end{equation}
Thus, if we can solve (\ref{eqPDE2}) then we also get the solution of (\ref{eqPDE}).
The advantage of the equation (\ref{eqPDE2}) is that it can be handled with the
method of characteristics (see \cite{char}).

First we translate (\ref{eqPDE2}) into a system of first order ordinary differential
equations:
\begin{equation}\label{eqPDE3}
\frac{dz}{dt} = A(z,w), \quad \frac{dw}{dt} = B(z,w), \quad 
\frac{dH}{dt} = C(z,w)H + D(z,w;v),
\end{equation}
where $z=z(t)$, $w = w(t)$, $H = H(t)$ are functions in $t$. A characteristic of
(\ref{eqPDE3}) is a function $F(z,w,H)$ for which we have 
$Q(z(t),w(t),H(t)) = const.$ Clearly, every characteristic $Q$ is a solution of 
(\ref{eqPDE2}). It is well known that a system of three equations has
two independent characteristics $Q_1,Q_2$ as a basis and every characteristic 
$Q$ can be expressed as $Q = F(Q_1,Q_2)$ for an arbitrary (differentiable) function $F$.
In the present case we have to solve the equation (\ref{eqQHequ}) which simplifies 
the situation. More precisely we can rewrite (\ref{eqQHequ}) to an equation of the
form 
\begin{equation}\label{eqQHequ2}
Q_2(z,w,H) = \tilde F(Q_1(z,w,H)),
\end{equation}
where $\tilde F$ is an arbitrary (differentiable) function.

In order to calculate two independent characteristics it is convenient to {\it eliminate}
$t$ from the system (\ref{eqPDE3}) which gives rise to a simpler system of
differential equation:
\begin{equation}\label{eqPDE4}
\frac{dz}{dw} = \frac{A(z,w)}{B(z,w)}, \quad 
\frac{dH}{dw} = \frac{C(z,w)}{B(z,w)}H + \frac{D(z,w;v)}{B(z,w)},
\end{equation}
where $z=z(w)$ and $H = H(w)$ are now considered as functions is $w$.

Let $z = f(c_1,w)$ be a one-parametric solution of the differential 
equation $\frac{dz}{dw} = \frac{A(z,w)}{B(z,w)}$, where $c_1$ is, for example,
the initial value $c_1 = z(0)$. If we express $c_1$ from the expression $z = f(c_1,w)$, 
that is, $c_1 = Q_1(z,w)$ then $Q_1$ is a characteristic of the system (\ref{eqPDE3}).
Note that $Q_1$ does not depend on $H$ and also not on $v$. Actually $Q_1$ just solves
the equation $A(z,w) Q_z + B(z,w) Q_w = 0$. Nevertheless it is a non-trivial characteristic
of (\ref{eqPDE3}).

In order to obtain a second characteristic we have to solve the second equation of 
(\ref{eqPDE4}) which is a first order linear differential equation. Note that we can
substitute $z = f(c_1,w)$ and obtain as a solution
\[
H = \exp\left( \int_0^w \frac{C(f(c_1,s),s)}{B(f(c_1,s),s)}\, ds \right) 
\left(\int_0^w \frac{D(f(c_1,s),s;v)}{B(f(c_1,s),s)} 
\exp\left( -\int_0^s \frac{C(f(c_1,t),t)}{B(f(c_1,t),t)}\, dt \right)ds  
 + c_2 \right),
\]
where $c_2$ is some constant. Again if we express $c_2$ explicitly (and eliminate $c_1$
with the help of $c_1 = Q_1(z,w)$) we get 
another characteristic:
\begin{eqnarray*}
c_2 &=& Q_2(z,w,H) \\
&=& H\exp\left(-\int_0^w \frac{C(f(Q_1(z,w),s),s)}{B(f(Q_1(z,w),s),s)}\, ds \right) \\
&-&\int_0^w \frac{D(f(Q_1(z,w),s),s;v)}{B(f(Q_1(z,w),s),s)} 
\exp\left( -\int_0^s \frac{C(f(Q_1(z,w),t),t)}{B(f(Q_1(z,w),t),t)}\, dt \right)ds.
\end{eqnarray*}
Now if we apply (\ref{eqQHequ2}) we obtain the following representation for $H$:
\begin{eqnarray*}
H &=& \exp\left(\int_0^w \frac{C(f(Q_1(z,w),s),s)}{B(f(Q_1(z,w),s),s)}\, ds \right) 
\\
&&\times \left(  \int_0^w \frac{D(f(Q_1(z,w),s),s;v)}{B(f(Q_1(z,w),s),s)} 
\exp\left( -\int_0^s \frac{C(f(Q_1(z,w),t),t)}{B(f(Q_1(z,w),t),t)}\, dt \right)ds  
+ \tilde F(Q_1(z,w)) 
\right).
\end{eqnarray*}
In our context we will assume that (\ref{eqHw0}) holds, that is, $H(z,0;v) = 0$,
which implies that $\tilde F( x ) = 0$. 
Consequently we have
\begin{eqnarray}
H(z,w;v) &=& \exp\left(\int_0^w \frac{C(f(Q_1(z,w),s),s)}{B(f(Q_1(z,w),s),s)}\, ds \right) 
\label{eqHrep}\\
&&\times \left(  \int_0^w \frac{D(f(Q_1(z,w),s),s;v)}{B(f(Q_1(z,w),s),s)} 
\exp\left( -\int_0^s \frac{C(f(Q_1(z,w),t),t)}{B(f(Q_1(z,w),t),t)}\, dt \right)ds  
\right).  \nonumber
\end{eqnarray}

\section{Singularity analysis}

Next we assume that the assumptions of Theorem~\ref{Th1} are satisfied so that we
can analyze the analytic properties of the solution function $H(z,w;v)$ that is given
by (\ref{eqHrep}). Actually we will show that if $v$ is close to $1$ that the dominant
singularity comes from a curve that is a pertubation of the curve $z+w = 1$.

First we note that by assumption the function $f(Q_1(z,w),s)$ is regular as well as
the fraction $C(z,w)/B(z,w)$. Consequently the function
\[
(z,w)\mapsto K(z,w) = \int_0^w \frac{C(f(Q_1(z,w),s),s)}{B(f(Q_1(z,w),s),s)}\, ds
\]
is analytic, too. Thus, it remains to consider the integral
\begin{eqnarray*}
&&\int_0^w \frac{D(f(Q_1(z,w),s),s;v)}{B(f(Q_1(z,w),s),s)} 
\exp\left( - K(z,s) \right)ds  \\
&&= \int_0^w \frac{a(f(Q_1(z,w),s),s;v)\exp\left( - K(z,s) \right)/B(f(Q_1(z,w),s),s) }
{(1 - b(f(Q_1(z,w),s),s;v))^2} \, ds.
\end{eqnarray*}
First let us assume that $v=1$. In this case we know by assumption that
$H(z,w;1) = 1/(1-z-w) - 1/(1-z)$. Furthermore we have $b(z,w;1) = z$.
Thus the above integral simplifies to
\[
 \int_0^w \frac{L(z,w,s)}{(1- f(Q_1(z,w),s))^2}\ ds,
\]
where $L(z,w,s)$ is a non-zero regular function. As long as $f(Q_1(z,w),s) \ne 1$
for $0\le s\le w$ then the integral represents a regular function in $z$ and $w$. 
Hence, we have to detect $s$ for which
$f(Q_1(z,w),s) = 1$. Let us first assume that $z$ and $w$ are real and positive. 
We also recall that by assumption $\frac{\partial f}{\partial s} = A(f,s)/B(f,s)< 0$. 
Thus, if we start with $z,w$ close
to zero and increase them we observe that the first critical instance occurs when
$f(Q_1(z,w),0) = 1$. Of course this has to coincide with the condition $z+w = 1$
and we have to recover the (known) singular behaviour $1/(1-z-w)$.

Actually we can use the following easy lemma (which follows from partial integration).

\begin{lemma}
Suppose that $N(s)$ and $D(s)$ are three times continuously
differentiable functions such that $D(s)\ne 0$ and $D'(s)\ne 0$
Then we have
\begin{eqnarray*}
\int \frac{ N(s)}{D(s)^2}\, ds &=& -\frac{N(s)}{D(s)D'(s)}
+ \frac{\log D(s)}{D'(s)} \left( \frac{N(s)}{D'(s)}\right)' \\
&-& \int \log D(s) \left( \frac 1{D'(s)} \left( \frac{N(s)}{D'(s)}\right)'  \right)' \, ds.
\end{eqnarray*}
\end{lemma}

If we apply this lemma in our context it follows that
\[
 \int_0^w \frac{L(z,w,s)}{(1- f(Q_1(z,w),s))^2}\ ds = 
 \frac{\tilde L_1(z,w)} {1- f(Q_1(z,w),0)} + O\left( \log|1- f(Q_1(z,w),0)|\right)
\]
for positive real $z,w$ with $z+w\to 1$ (and a proper non-zero analytic function $\tilde L_1(z,w)$).
Summing up we obtain for positive real $z,w$ with $z+w\to 1$ the asymptotic representation
\[
H(z,w;1) = \frac{\tilde L_2(z,w)} {1- f(Q_1(z,w),0)} + O\left( \log|1- f(Q_1(z,w),0)|\right)
\]
for some non-zero analytic function $\tilde L_2(z,w)$. In particular it follows that
$1- f(Q_1(z,w),0)$ can be written as
\[
1- f(Q_1(z,w),0) = \tilde L_2(z,w) (1-z-w).
\]
Of course the same kind of analysis applies if $z$ and $w$ are complex numbers close to the positive real
line. Furthermore we observe that the integral representation for $H(z,w;1)$ will not get singular
if $z+w\ne 1$. By continuity this also holds if $v$ is close to $1$ and $|1-z-w|\ge \delta$ for
some $\delta > 0$. 

Finally if $v$ is close (but different) to $1$ and $z$ and $w$ satisfy $|1-z-w| < \delta$ 
then we just have to modify the above analysis slightly and observe that 
$H(z,w;v)$ can be represented as
\[
H(z,w;v) = \frac{\tilde L_2(z,w;v)} {1- b(f(Q_1(z,w),0),0;v)} + O\left( \log|1- b(f(Q_1(z,w),0),0;v)|\right).
\]
Thus, the equation
\begin{equation}\label{eqsingcurve}
b(f(Q_1(z,w),0),0;v) = 1
\end{equation}
determines the dominant singularity of $H(z,w;v)$. By the implicit function theorem it follows 
that there exists a solution of (\ref{eqsingcurve}) of the form $z=z_0(w;v)$ with
$z_0(w;1) = 1-w$ (if $w$ is close to the positive real line segment $[0,1]$).

\section{A central limit theorem}

We start with a lemma on bivariate asymptotics for generating functions
in two variables which is a slight generalization of the smooth case
in Pemantle and Wilson's book \cite{PW}.

\begin{lemma}
Suppose that $f(z,w)$ is a generating function in two variables that can
be written in the form 
\[
f(z,w) = \frac{N(z,w)}{D(z,w)},
\]
where $N$ and $D$ are regular functions such that the system of equations
\begin{equation}\label{eqsaddle}
D(z,w) = 0, \quad w D_w(z,w) = \rho z D_z(z,w)
\end{equation}
has a unique positive and analytic solution $z = z_0(\rho)$, $w = w_0(\rho)$ for $\rho$
in a positive interval $[\alpha,\beta]$ such that $D_w(z_0(\rho),w_0(\rho)) \ne 0$ in this range 
and that $D(z,w) = 0$ has no other solutions for $|z|\le z_0(\rho)$, $|w|\le w_0(\rho)$.
Furthermore we assume that $N(z_0(\rho),w_0(\rho)) \ne 0$.

Then we have uniformly for $m/n \in [\alpha,\beta]$
\begin{equation}\label{eqdoubleasymptotics}
[z^nw^m] f(z,w) \sim \frac{N(z_0(m/n),w_0(m/n))}{- z_0(m/n)w_0(m/n) D_z(z_0(m/n),w_0(m/n))} 
\frac {z_0(m/n)^{-n}w_0(m/n)^{-m}}{\sqrt{2\pi n \Delta(m/n) }},
\end{equation}
where
\[
\Delta(\rho) = \left.\frac{D_{zz} D_w^2 - 2D_{zw} D_z D_w + D_{ww} D_z^2}{z D_z^3}
+ \frac{D_w^2}{z^2 D_z^2} + \frac{D_w}{zw D_z}\right|_{z=z_0(\rho),w=w_0(\rho)}.
\]
\end{lemma}

\begin{proof}
By assumption the map $z\mapsto f(z,w)$ has a unique polar singularity at $z= z(w)$,
where $z(w)$ is determined by $D(z(w),w) = 0$ (for $w$ close to the real interval
$[w_0(a),w_0(b)]$) which implies 
\[
[z^n] f(z,w) \sim  \frac{N(z(w),w)}{-z(w)D_z(z(w),w)} z(w)^{-n}.
\]
Finally we fix the ratio $m/n= \rho$ and a direct application of the saddle point method
on the Cauchy integral evaluating
\[
[w^m z^n] f(z,w) = \frac 1{2\pi i} \int_{|w| = w_0(\rho)} \left([z^n] f(z,w) \right) w^{-m-1}\, dw
\]
leads to the result. Note that the saddle point $w = w_0(\rho)$ that comes from the 
power $z(w)^{-n} w^{-\rho n}$ has to satisfy (\ref{eqsaddle}).
\end{proof}

We now apply this procedure to a slightly more general situation, namely when there is
a further parameter $v$ (that is assumed to be close to $1$):
\[
f(z,w;v) = \frac{N(z,w;v)}{D(z,w;v)}.
\]
In our context we have to identify $f(z,w;v)$ with $H(z,w;v)$ and 
$D(z,w;v)$ with $1- b(f(Q_1(z,w),0),0;v)$.
Of course we have to formulate proper assumptions (similar to the above which 
are actually satisfied for $H(z,w;v)$) and, hence, by 
(\ref{eqdoubleasymptotics}) we obtain an asymptotic expansion of the form
\[
[z^nw^m] H(z,w;v) \sim \frac{C(m/n;v)}{\sqrt{2\pi n}} z_0(m/n;v)^{-n} w_0(m/n;v)^{-m}
\]
that is uniform in $v$ (for $v$ sufficiently close to $1$).

If we fix the ratio $\rho = m/n$ the leading asymptotics is then just a power in $n$:
\[
z_0(\rho;v)^{-n} w_0(\rho;v)^{-\rho n} = e^{h(\rho;v) n }
\]
with $h(\rho;v) = -\log z_0(\rho;v) - \rho \log w_0(\rho;v)$.
Actually we have a so-called {\it quasi-power}, where we can expect that
(after proper normalization) a central limit theorem should hold.

In our context we obtain
\[
\mathbb{E}[ v^{X_{\rho n,n}}] = \frac{[z^n w^{\rho n}] H(z,w;v)}{[z^n w^{\rho n}] H(z,w;1) }
\sim \frac{C(\rho;v)}{C(\rho;1)} \left( \frac{z_0(\rho;1) w_0(\rho;1)^\rho}{z_0(\rho;v) w_0(\rho;v)^\rho}\right)^n.
\]
And this is precisely the assumption that is needed in order to apply
Hwang's {\it Quasi-Power Theorem} \cite{Hwang}.

\begin{lemma}
Let ${X}_n$ be a random variable with the property that
\begin{equation}\label{eqTh2A1}
\mathbb{E}\, {v}^{{X}_n} =  e^{\lambda_n\cdot A(v)+B(v)}
\left( 1 + O\left( \frac 1{\varphi_n} \right) \right)
\end{equation}
holds uniformly in a complex neighbourhood of ${v} ={1}$,
where $\lambda_n$ and $\varphi_n$ are sequences of positive real
numbers with $\lambda_n\to\infty$ and $\varphi_n\to\infty$, and
$A({v})$ and $B(v)$ are analytic functions in this
neighbourhood of ${v} ={1}$ with $A({1}) =B({1}) =
{0}$.
Then ${X}_n$ satisfies a central limit
theorem  of the form
\begin{equation}\label{eqPro32.2}
\frac 1{\sqrt{\lambda_n}} \left( {X}_n - \mathbb{E}\, {X}_n \right) \to
N\left( {0},\sigma^2 \right)
\end{equation}
and we have
\[
\mathbb{E}\, {X}_n = \lambda_n \mu  + O\left(
1+\lambda_n/\varphi_n \right)
\]
and 
\[
\mathbb{V}{\rm ar}\, X_n = \lambda_n {\sigma^2} + O\left( \left( 1+\lambda_n/\varphi_n \right)^2 \right),
\]
where 
\[
\mu = A'({1}) 
\]
and
\[
\sigma^2 = A''({1}) + A'(1).
\]
\end{lemma}

Recall that  $A(v) = h(\rho;v) = -\log z_0(\rho;v) - \rho \log w_0(\rho;v)$ so that
\[
\mu = \mu(\rho) = - \frac{z_{0,v}(\rho;1)}{z_{0}(\rho;1)} - \rho  \frac{w_{0,v}(\rho;1)}{w_{0}(\rho;1)}.
\]
Since we have assumed that $X_{n,m}$ are non-negative random variables we can only expect
a central limit theorem if $\mu > 0$, since for $\mu = 0$ it would follow that $X_{n,m}$ is
negative with probability $1/2$.

Finally we mention that since the convergence is uniform in $\rho\in [a,b]$ we also get 
a central limit theorem for $n,m\to\infty$ if $m/n\in [a,b]$. 
This completes the proof of our main Theorem~\ref{Th1}.

\acknowledgements
\label{sec:ack}
The authors are grateful to three anonymous referees for their careful reading and their valuable comments.


\end{document}